\documentclass[a4paper,12pt]{article}

\makeatletter
\renewcommand\title[1]{\gdef\@title{\reset@font\Large\bfseries #1}}
\renewcommand\section{\@startsection {section}{1}{\z@}%
				{-3.5ex \@plus -1ex \@minus -.2ex}%
				{2.3ex \@plus.2ex}%
				{\normalfont\large\bfseries\boldmath}}
\renewcommand\subsection{\@startsection{subsection}{2}{\z@}%
				{-3ex\@plus -1ex \@minus -.2ex}%
				{1.5ex \@plus .2ex}%
				{\normalfont\normalsize\bfseries\boldmath}}
\renewcommand\subsubsection{\@startsection{subsubsection}{3}{\z@}%
				{-2.5ex\@plus -1ex \@minus -.2ex}%
				{1.5ex \@plus .2ex}%
				{\normalfont\normalsize\bfseries\boldmath}}

\renewcommand\paragraph{\@startsection{paragraph}{4}{\z@}%
				{2ex \@plus.5ex \@minus.2ex}%
				{-1em}%
				{\normalfont\normalsize\bfseries}}

\renewcommand\subparagraph{\@startsection{subparagraph}{5}{\parindent}%
				{2ex \@plus.5ex \@minus .2ex}%
				{-1em}%
				{\normalfont\normalsize\bfseries}}
\makeatother

\usepackage{geometry}
\usepackage[final]{microtype}
\usepackage{ifxetex}
\ifxetex%
	\usepackage{amsthm,amsmath,amssymb}
	\makeatletter
		\newcommand{\blx@nowarnpolyglossia}{}
	\makeatother

	\usepackage{csquotes} %
	\usepackage{polyglossia}
	\setdefaultlanguage{english}
\else
	\RequirePackage[utf8]{inputenc}

	\usepackage{csquotes}

	\RequirePackage{amsfonts}
	\RequirePackage{amsthm,amsmath}
	\RequirePackage{amssymb}

	\usepackage[english]{babel}
\fi
\usepackage[unicode=true,
	colorlinks=true,
	citecolor=black,
	linkcolor=black,
	urlcolor=blue,
	pdftitle={Half-graphs, other non-stable degree sequences, and the switch
	Markov chain},
	pdfauthor={P\'eter L. Erd\H{o}s, Ervin Gy\H{o}ri, Tam\'as R\'obert
	Mezei, Istv\'an Mikl\'os, D\'aniel Solt\'esz},
	pdfkeywords={half-graph, switch, Markov chain, P-stable, rapid mixing},
	pdfsubject={Switch (swap) Markov chain},
	]{hyperref}

\usepackage{tikz}
\usetikzlibrary{decorations.pathreplacing,arrows}
\usepackage{bm}

\usepackage{dsfont}
\usepackage{float}

\usepackage{rotating}
\usepackage{graphicx}
\usepackage{subfigure}

\usepackage{enumerate}

\newtheorem{theorem}{Theorem}[section]

\newtheorem{conjecture}[theorem]{Conjecture}

\newtheorem{lemma}[theorem]{Lemma}
\newtheorem{definition}[theorem]{Definition}

\newtheorem{question}[theorem]{Question}

\newtheorem{corollary}[theorem]{Corollary}

\theoremstyle{remark}

\newcommand{\bN}{\mathbb{N}}
\newcommand{\bS}{\mathbb{S}}
\newcommand{\bZ}{\mathbb{Z}}
\newcommand{\cD}{\mathcal{D}}
\newcommand{\cG}{\mathcal{G}}
\newcommand{\cH}{\mathcal{H}}
\newcommand{\cM}{\mathcal{M}}
\newcommand{\cO}{\mathcal{O}}
\newcommand{\cP}{\mathcal{P}}
\newcommand{\cT}{\mathcal{T}}

\title{Half-graphs, other non-stable degree sequences, and the switch Markov
chain}
\author{%
	Péter L. Erdős\footnote{These authors were supported in part by the National
	Research, Development and Innovation Office, NKFIH grants K-116769 and K-132930}
	\textsuperscript{,}\footnote{These authors were supported in part by the National
	Research, Development and Innovation Office, NKFIH grant KH-126853} \qquad
	Ervin Győri\textsuperscript{$*$} \qquad
	Tamás Róbert Mezei\textsuperscript{$*$,$\dagger$} \\
	István Miklós\textsuperscript{$\dagger$,}\footnote{IM was supported in part by
	the National Research, Development and Innovation Office, NKFIH grant
	SNN-116095} \qquad
	Dániel Soltész\textsuperscript{$\dagger$,}\footnote{DS was supported in part by
	the National Research, Development and Innovation Office, NKFIH grants
	K-120706 and KH-130371.}\\
	\small Alfréd Rényi Institute of Mathematics\\[-0.8ex]
	\small Reáltanoda street 13--15,\\[-0.8ex]
	\small H-1053 Budapest, Hungary \\
	\small\tt $<$erdos.peter,gyori.ervin,mezei.tamas.robert,\\[-0.8ex]
	\small\tt miklos.istvan,soltesz.daniel$>$@renyi.hu
}
\date{%
	\small Alfréd Rényi Institute of Mathematics (a Hungarian Academy
	of Sciences Centre of Excellence), Reáltanoda~u.\ 13--15, 1053
	Budapest, Hungary \\
	\vspace{10pt}
	\normalsize \today
	\vspace{-12pt}
}

\begin{document}

\maketitle

\begin{abstract}
	One of the simplest methods of generating a random graph with a given
	degree sequence is provided by the Monte Carlo Markov Chain method using
	\textbf{switches}. The switch Markov chain converges to the uniform
	distribution, but generally the rate of convergence is not known. After
	a number of results concerning various degree sequences, rapid mixing
	was established for so-called $P$-stable degree sequences (including
	that of directed graphs), which covers every previously known rapidly
	mixing region of degree sequences.

	\medskip

	In this paper we give a non-trivial family of degree sequences that are
	not $P$-stable and the switch Markov chain is still rapidly mixing on
	them. This family has an intimate connection to
	Tyshkevich-decompositions and strong stability as well.
\end{abstract}

\section{Introduction}

\subsection{Previous results on the rapid mixing of the switch Markov chain}
An important problem in network science is to sample simple graphs with a given
degree sequence (almost) uniformly. In this paper we study a \textbf{Markov
Chain Monte Carlo} (MCMC) approach to this problem. The MCMC method can be
successfully applied in many special cases. A vague description of this
approach is that we start from an arbitrary graph with a given degree sequence
and sequentially apply small random modifications that preserve the degree
sequence of the graph. This can be viewed as a random walk on the space of
\textbf{realizations} (graphs) of the given degree sequence. It is well-known
that after sufficiently many steps the distribution over the state space is
close to the uniform distribution. The goal is to prove that the necessary
number of steps to take (formally, the mixing time of the Markov chain) is at
most a polynomial of the length of the degree sequence.

In this paper we study the so-called \textbf{switch Markov chain} (also known as
the swap Markov chain). For clarity, we refer to the degree sequence of a
simple graph as an \textbf{unconstrained} degree sequence.

Throughout the paper, we work with finite graphs and finite degree sequences.
For two graphs $G_1,G_2$ on the same labelled vertex set, we define their
\textit{symmetric difference} $G_1 \triangle G_2$ with $V(G_1 \triangle
G_2)=V(G_1)=V(G_2)$ and $E(G_1 \triangle G_2)= E(G_1) \triangle E(G_2)$.

\begin{definition}[switch]\label{def:switch}
	For a bipartite or an unconstrained degree sequence $\bm{d}$, we say
	that two realizations $G_1, G_2 \in \cG(\bm{d})$ are connected
	by a \textit{switch}, if \[|E(G_1 \triangle G_2)|=4.\]
\end{definition}

A widely used alternative name for switch is \emph{swap}. A switch (swap) can be
seen in Figure~\ref{fig:switch}; for the precise definition of the switch Markov
chain, see Definition~\ref{def:ucChain}. Clearly, if $G_1$ and $G_2$ are two
simple graphs joined by a switch, then $F=E(G_1) \triangle E(G_2)$ is a cycle of
length four (a $C_4$), and $E(G_2)=E(G_1)\Delta F$. Hence, the term switch is
also used to refer to the operation of taking the symmetric difference with a
given $C_4$. It should be noted, though, that only a minority of $C_4$'s define
a (valid) switch. The majority of $C_4$'s do not preserve the degree sequence
(if the $C_4$ does not alternate between edges of $G_1$ and $G_2$), or they
introduce an edge which violates the constraints of the model (say, an edge
inside one of the color classes in the bipartite case).

\begin{figure}
\centering
\begin{tikzpicture}[scale=1.5]
\draw[color=gray] (1,1) edge[bend left=30] (0.4,0.6);
\draw[color=gray] (1,1) (1,1) edge[bend left=30] (0.5,0.5);
\draw[color=gray] (1,1) (1,1) edge[bend left=30] (0.6,0.4);
\draw[color=gray] (1,1) (1,2) edge[bend left=20] (0.4,2.5);
\draw[color=gray] (1,2) edge[bend right=30] (0.4,1.8);
\draw[color=gray] (2,2) edge[bend right=30] (2.5,2.5);
\draw[thick] (1,1) -- (1,2);
\draw[thick] (2,1) -- (2,2);
\draw[thick,densely dashed] (1,1) -- (2,1);
\draw[thick,densely dashed] (1,2) -- (2,2);

\draw [thick,->] (3,1.5) -- (3.75,1.5);

\draw[color=gray] (5,1) edge[bend left=30] (4.4,0.6);
\draw[color=gray] (5,1) edge[bend left=30] (4.5,0.5);
\draw[color=gray] (5,1) edge[bend left=30] (4.6,0.4);
\draw[color=gray] (5,2) edge[bend left=20] (4.4,2.5);
\draw[color=gray] (5,2) edge[bend right=30] (4.4,1.8);
\draw[color=gray] (6,2) edge[bend right=30] (6.5,2.5);
\draw[thick,densely dashed] (5,1) -- (5,2);
\draw[thick,densely dashed] (6,1) -- (6,2);
\draw[thick] (5,1) -- (6,1);
\draw[thick] (5,2) -- (6,2);

\foreach \x in {1,2}
\foreach \y in {1,2}
{%
	\node[draw,fill,circle,inner sep=2pt,minimum size] at (\x,\y) {};
	\node[draw,fill,circle,inner sep=2pt,minimum size] at (\x+4,\y) {};
}

\end{tikzpicture}
\caption{A switch (dashed lines emphasize missing edges)}\label{fig:switch}
\end{figure}
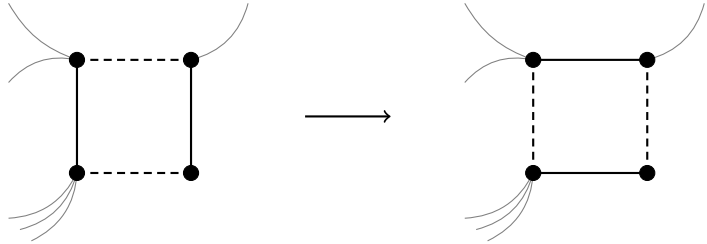

The question whether the mixing time of the switch Markov chain is short enough
is interesting from both a practical and a theoretical point of view (although
short enough depends greatly on the context). The switch Markov chain is already
used in applications, hence rigorous upper bounds on its mixing time are much
needed, even for special cases.

The switch Markov chain uses transitions which correspond to minimal
perturbations.  There are many other instances where the Markov chain of the
smallest perturbations have polynomial mixing time,
see~\cite{sinclair_algorithms_1993}. However, it is unknown whether the mixing
time of the switch Markov chain is uniformly bounded by a polynomial for every
(unconstrained) degree sequence. Hence from a theoretical point of view, even an
upper bound of $\cO(n^{10})$ on the mixing time of the switch Markov chain would
be considered a great success, even though in practice it is only slightly
better than no upper bound at all.

The present paper is written from a theoretical point of view and should be
considered as a step towards answering the following question.
\begin{question}[Kannan, Tetali, and Vempala~\cite{kannan_simple_1999}]
	Is the switch Markov chain rapidly mixing on the realizations of all
	graphic degree sequences?
\end{question}

$P$-stability was introduced by Jerrum and Sinclair for the Jerrum-Sinclair
chain, for whose rapid mixing the notion presents a natural
boundary~\cite{stefankovic_counting_2018}. Jerrum, Sinclair, and
McKay~\cite{jerrum_when_1989} already give an example for a non-$P$-stable
degree sequences which has a unique realization (trivially rapidly mixing): take
\begin{equation}\label{eq:JSMdegseq}
	(2n-1,2n-2,\ldots, n+1,n,n,n-1,\ldots,2,1)\in\bN^{2n}.
\end{equation}
In its unique realization, the first $n$ vertices form a clique, while the
remaining vertices form an independent set.

We will denote graphs with upper case letters (e.g.\ $G$), degree sequences
(which are non-negative integer vectors) with bold lower case letters (e.g.
$\bm{d}$).  Classes of graphs and classes of degree sequences are both denoted
by upper case calligraphic letters (e.g.  $\cH$). We say that a graph $G$ is a
realization of a degree sequence $\bm{d}$, if the degree sequence of $G$ is
$\bm{d}$. For a degree sequence $\bm{d}$, we denote the set of all realizations
of $\bm{d}$ by $\cG(\bm{d})$. The $\ell^1$-norm of a vector $x$ is denoted by
$\|x\|_1$.

\begin{definition}[Greenhill~and~Gao~\cite{gao_mixing_2020}]\label{def:Pstable}
	Let $\cD$ be a set of graphic degree sequences and $k\in 2\bN$. We say
	that $\cD$ is \textbf{$\bm{k}$-stable}, if there exists a polynomial
	$p\in \mathbb{R}\left[x\right]$ such that for any $n\in \bN$ and any
	degree sequence $\bm{d} \in \cD$ on $n$ vertices, any degree sequence
	$d'$ with $\|d'-d\|_1\le k$ satisfies $\left|\cG(\bm{d}')\right|\le
	p(n)\cdot \left|\cG(\bm{d})\right|$. The term \textbf{$\bm{P}$-stable}
	is an alias for $2$-stable, which is the least restrictive non-trivial
	class defined here.
\end{definition}

There is a long line of results where the rapid mixing of the switch Markov
chain is proven for certain degree sequences, see~\cite{cooper_sampling_2007,
miklos_towards_2013, greenhill_polynomial_2011, erdos_approximate_2015,
erdos_efficiently_2018, greenhill_switch_2018}.  Some of these results were
unified, first by Amanatidis and Kleer~\cite{amanatidis_rapid_2019}, who
established rapid mixing for so-called \emph{strongly stable} classes of degree
sequences of simple and bipartite graphs (definition given in
Section~\ref{sec:sstable}).

The most general result at the time of
writing is proved by Erdős, Greenhill, Mezei, Miklós, Soltész, and Soukup:
\begin{theorem}[\cite{erdos_mixing_2019}]\label{thm:rapidPstable}
	The switch Markov chain is rapidly mixing on sets of unconstrained,
	bipartite, and directed degree sequences that are
	\textbf{$\bm{P}$-stable} (see Definition~\ref{def:Pstable}).
\end{theorem}
For the sake of being less redundant, the phrase ``$\cD$ is rapidly mixing''
shall carry the same meaning as ``switch Markov chain is rapidly mixing on
$\cD$''.

Recently, Greenhill and Gao~\cite{gao_mixing_2020} presented elegant conditions
which when satisfied ensure $8$-stability of a class of degree sequences
($8$-stable degree sequence are by definition $P$-stable). In particular, they
show that for $\gamma>2$, power-law distributed degree sequences are
$8$-stable, hence rapidly mixing. They also give a proof that $8$-stable sets
of degree sequence are rapidly mixing.

In this paper we try to extend the set of rapidly mixing bipartite degree
sequences beyond $P$-stability. The degree sequence~\eqref{eq:JSMdegseq} can
naturally be turned into a bipartite one by assigning the role of the two color
classes to the clique and the independent set, and then removing the edges of
the clique.
\begin{definition}
	Let us define a bipartite degree sequence:
	\begin{align*}
		\bm{h}_0(n)&:=
		\left(
			\begin{array}{ccccccc}
			1 & 2 & 3 & \cdots & n-2 & n-1 & n \\
			n & n-1 & n-2 & \cdots & 3 & 2 & 1 \\
			\end{array}
		\right)\\
		\cH_0&:=\left\{\bm{h}_0(n)\ \big|\ n\in\bN\right\}
	\end{align*}
	Let $A_n=\{a_1,\ldots,a_n\}$ and $B_n=\{b_1,\ldots,b_n\}$, often denoted
	simply $A$ and $B$. We label the vertices of $\bm{h}_0(n)$ such that $A$
	is the first and $B$ is the second color class, with
	$\deg_{\bm{h}_0(n)}(a_i)=n+1-i$ and $\deg_{\bm{h}_0(n)}(b_i)=i$ for
	$i\in [1,n]$. The unique realization $H_0(n)$, also known as the
	half-graph, is displayed on Figure~\ref{fig:halfgraph}.

	\begin{figure}[H]
	\centering
	\begin{tikzpicture}[scale=1.7]
	\tikzstyle{vertex}=[draw,circle,fill=black,minimum size=3,inner sep=0]
	\node[vertex] (y5) at (0,0) [label=south:{$a_1$}] {};
	\node[vertex] (y4) at (1,0) [label=south:{$a_2$}] {};
	\node[vertex] (y3) at (4,0) [label=south:{$a_i$}] {};
	\node[vertex] (y2) at (7,0) [label=south:{$a_{n-1}$}] {};
	\node[vertex] (y1) at (8,0) [label=south:{$a_n$}] {};
	\node[vertex] (x5) at (8.65,1) [label=north:{$b_{n}$}] {};
	\node[vertex] (x4) at (7.65,1) [label=north:{$b_{n-1}$}] {};
	\node[vertex] (x3) at (4.65,1) [label=north:{$b_{i}$}] {};
	\node[vertex] (x2) at (1.65,1) [label=north:{$b_2$}] {};
	\node[vertex] (x1) at (0.65,1) [label=north:{$b_1$}] {};

	\draw (y5)--(x1);
	\draw (y5)--(x2);
	\draw (y5)--(1,0.47);
	\draw (y5)--(1,0.35);
	\draw[dashed] (y4)--(x1);
	\draw (y4)--(x2);
	\draw (y4)--(2,0.5);
	\draw (y4)--(2,0.4);
	\draw[dashed] (x1)--(1.35,0.7);
	\draw[dashed] (x1)--(1.3,0.5);
	\draw[dashed] (x2)--(2.35,0.7);
	\draw[dashed] (x2)--(2.3,0.5);

	\draw (x3)--(y3);
	\draw[dashed] (x3)--(5.2,0.35);
	\draw[dashed] (x3)--(5.3,0.6);
	\draw[dashed] (x3)--(5.35,0.75);
	\draw (x3)--(4,0.4);
	\draw (x3)--(3.95,0.6);
	\draw (x3)--(3.9,0.7);
	\draw (y3)--(4.6,0.3);
	\draw (y3)--(4.5,0.6);
	\draw (y3)--(4.55,0.45);
	\draw[dashed] (y3)--(3.5,0.6);
	\draw[dashed] (y3)--(3.45,0.45);
	\draw[dashed] (y3)--(3.4,0.3);

	\draw (x5)--(y1);
	\draw (x5)--(y2);
	\draw (x5)--(8,0.7);
	\draw (x5)--(8,0.8);
	\draw[dashed] (x4)--(y1);
	\draw (x4)--(y2);
	\draw (x4)--(7,0.7);
	\draw (x4)--(7,0.6);
	\draw[dashed] (y1)--(7.5,0.7);
	\draw[dashed] (y1)--(7.5,0.5);
	\draw[dashed] (y2)--(6.4,0.75);
	\draw[dashed] (y2)--(6.4,0.5);
	\end{tikzpicture}
	\caption{The unique realization $H_0(n)$ of $\bm{h}_0(n)$ is isomorphic
	to the half-graph.}%
	\label{fig:halfgraph}
	\end{figure}
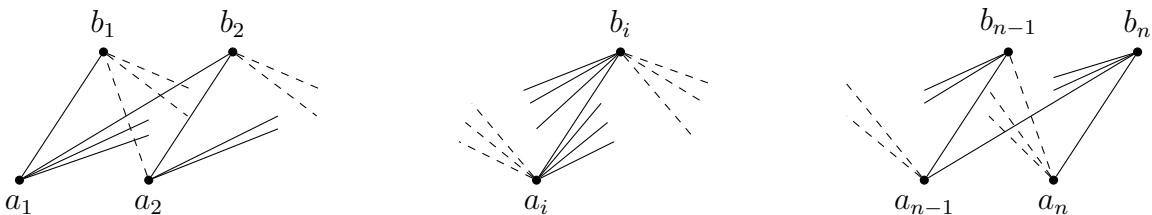
\end{definition}
In this paper, we conduct a detailed study of $\bm{h}_0(n)$ and its
neighborhood. Before presenting our main results, let us get familiar with two
interesting properties of $\bm{h}_0(n)$.

\subsection{Simple examples for rapidly mixing non-stable bipartite
classes}\label{sec:examples}

Let $\mathds{1}_{x}$ be the vector which takes 1 on $x$ and zero everywhere
else.  As solving an easy linear recursion in
Corollary~\ref{cor:orderofmagnitude} shows,
$\bm{h}_0(n)-\mathds{1}_{a_1}-\mathds{1}_{b_n}$ has
\begin{equation*}
	\varTheta\left({\left(\frac{1+\sqrt{5}}{2}\right)}^n\right)
\end{equation*}
realizations, therefore $\cH_0$ is not $P$-stable.

Although $\bm{h}_0(n)$ seems very pathological as an example for a non-stable
degree sequence, it is a source of more interesting examples. As pointed out to
us by an anonymous reviewer, one may replace $a_{i}b_{i}$ by a pair of
independent edges simultaneously for all $i$: let us define
\begin{align*}
	\bm{g}(n):=
	\left(
		\begin{array}{ccccccc}
		1 & 1 & 3 & 3 & \cdots & 2n-1 & 2n-1 \\
		2n-1 & 2n-1 & 2n-3 & 2n-3 & \cdots & 1 & 1 \\
		\end{array}
	\right).
\end{align*}
The number of realizations of $\bm{g}(n)$ is $2^n$, because the two independent
edges replacing an edge $a_{i}b_{i}$ can be switched with the two induced
non-edges. In addition, every realization of $\bm{g}(n)$ can be obtained this
way, so for any realization of $\bm{g}(n)$ the previously mentioned $n$ switches
are the complete set of switches. Because the random-walk on a hypercube is
rapidly mixing, the switch Markov chain is rapidly mixing on $\{ \bm{g}(n)\ |\
n\in\bN\}$. Moreover, by solving yet another a linear recursion, one can verify
that $\{\bm{g}(n)\ |\ n\in\bN \}$ is not $P$-stable.

In Section~\ref{sec:compnonstable}, we will draw the curtain on the explanation
behind the behavior of $\bm{h}_0(n)$ and $\bm{g}(n)$. In the meantime, we
present the main results of the paper.

\subsection{Results}

If $\bm{d}$ is the degree sequence of the bipartite graph $G[A,B]$, then
$\bm{d}=(\bm{d}^A;\bm{d}^B)$ is split across the bipartition as well, and it is
called a splitted bipartite degree sequence. We say that $G[A,B]$ is the empty
bipartite graph if both $A=B=\emptyset$.

\begin{definition}
	For a set $\mathcal{D}$ of bipartite degree sequences, let
	\begin{align*}
		B_{2k}(\cD)&=\bigcup_{\bm{d}\in\cD}
		\left\{ \bm{e}:\operatorname{Dom}(\bm{d})\to\bN\ \Big|
		\ {\|d-e\|}_1\le 2k,\ \|e^A\|_1=\|e^B\|_1 \right\} \\
		\bS_{2k}(\cD)&=\bigcup_{\bm{d}\in\cD}
		\left\{ \bm{e}:\operatorname{Dom}(\bm{d})\to\bN\ \Big|
		\ {\|d-e\|}_1=2k,\ \|e^A\|_1=\|e^B\|_1 \right\}
	\end{align*}
	be the (closed) ball and sphere of radius $2k$ around $\cD$ (w.l.o.g.\
	$k\in \bN$). The requirement that $\|e^A\|_1=\|e^B\|_1$, i.e., that the
	sum of the degrees on the two sides be equal is necessary for
	graphicality.
\end{definition}

We will show in Section~\ref{sec:main} that neighborhoods of
$\cH_0=\left\{\bm{h}_0(n)\ |\ n\in\bN\right\}$ are rapidly mixing.
\begin{theorem}\label{thm:main}
	For any fixed $k$, the switch Markov chain is rapidly mixing on the
	bipartite degree sequences in $B_{2k}(\cH_0)$.
\end{theorem}

Next, we show that even though balls of constant size around $\cH_0$ are rapidly
mixing, $\bS_{2k}(\cH_0)$ contains a degree sequence which is not $P$-stable.

\begin{definition}
	For all $k,n\in\bN$ where $k<n$ let
	\begin{align*}
		\bm{h}_k(n)&:=\bm{h}_0(n)-k\cdot\mathds{1}_{a_1}-k\cdot\mathds{1}_{b_n}\\
		\quad\cH_k &:=
		\left\{\bm{h}_k{(n)}\ |\ k\le n\in\bN^+ \right\}
	\end{align*}
	be a bipartite degree sequence and a class of bipartite degree
	sequences, respectively.
\end{definition}

\begin{theorem}\label{thm:nonstable}
	The class of degree sequences $\cH_k$ is not $P$-stable for any
	$k\in\bN$.
\end{theorem}

\subsection{Outline}
The rest of the paper is organized as follows.
\begin{itemize}
	\item As promised at the end of Section~\ref{sec:examples}, we
		introduce the Tyshkevich-decomposition of bipartite graphs in
		Section~\ref{sec:tyshkevich}. We also expose a connection to
		strong stability which provides further motivation to studying
		$\bm{h}_0(n)$.
	\item In Section~\ref{sec:markov} we introduce the switch Markov chains, some
		related definitions, and Sinclair's result on mixing time.
	\item Section~\ref{sec:flow} describes the structure of realizations
		of degree sequences from $B_{2k}(\bm{h}_0)$, which is then used
		by Sections~\ref{sec:main}~and~\ref{sec:nonstable} to prove
		Theorems~\ref{thm:main}~and~\ref{thm:nonstable}, respectively.
	\item Section~\ref{sec:conclusion} describes how $\bm{h}_0(n)$ relates
		to previous research. Possible generalization of
		Theorem~\ref{thm:main} are conjectured.
\end{itemize}

\section{Properties of Tyshkevich-decompositions}\label{sec:tyshkevich}
\subsection{Tyshkevich-decomposition of bipartite graphs}\label{sec:bip}

Let $G$ be a simple graph. It is a split graph if there is a partition
$V(G)=A\uplus B$ ($A\neq \emptyset$ or $B\neq\emptyset$) such that $A$ is a
clique and $B$ is an independent set in $G$. Split graph were first studied by
Földes and Hammer~\cite{foldes_split_1977}, who determined that being split is a
property of the degree sequence $\bm{d}$ of $G$. Note, that the partition is
not necessarily unique, but the size of $A$ is determined up to a $+1$ additive
constant, see~\cite{hammer_splittance_1981}. A split graph endowed with the
partition is called a splitted graph, denoted by $(G,A,B)$. In addition
to~\cite{foldes_split_1977}, Tyshkevich and
Chernyak~\cite{tyshkevich_decomposition_1985} also determined that being split
is a property of the degree sequence, thus every realization of a split degree
sequence is a split graph.

Tyshkevich and co-authors have extensively studied a composition operator
$\circ$ on (split) graphs; these results are nicely collected
in~\cite{tyshkevich_decomposition_2000}. The composition $(G,A,B)\circ H$ takes
the disjoint union of split graph and a simple graph, and joins every vertex in
$A$ to every vertex of $H$. A fundamental result on this operator is that any
simple graph can be uniquely decomposed into the composition of split graphs
and possibly an indecomposable simple graph as the last factor. During the
writing of this paper, we were greatly saddened to learn that Professor
Tyshkevich passed away November, 2019

Let us slightly change the conventional notation $G[A,B]$ to also signal that
the color classes $A$ and $B$ are ordered (2-colored); to emphasize this, we may
refer to such graphs as \emph{splitted} bipartite graphs. Observe, that a
function $\Psi$ removing the edges of the clique on $A$ from $(G,A,B)$ produces
a splitted bipartite graph $G[A,B]$. Erdős, Miklós, and
Toroczkai~\cite{erdos_new_2018} adapted the results about split graphs and the
composition operator $\circ$ to splitted bipartite graphs via the bijection
given by $\Psi$.

\begin{definition}\label{def:bipcomp}
	Given two splitted bipartite graphs $G[A,B]$ and $H[C,D]$ with disjoint
	vertex sets, we define their (Tyshkevich-)\:composition $G[A,B]\circ
	H[C,D]$ as the bipartite graph
	\[
	G[A,B]\circ H[C,D]:=G[A,B]\cup H[C,D]+\{ ad\ |\ a\in A,\ d\in D \}.
	\]
\end{definition}

The $\circ$ operator is clearly associative, but not commutative. We say that a
bipartite graph is \emph{indecomposable} if it cannot be written as a
composition of two \emph{non-empty} bipartite graphs.

\begin{lemma}[\cite{erdos_new_2018}, adapted from Theorem
	2(i) in~\cite{tyshkevich_decomposition_2000}]\label{lem:bipindecomp}
	Let $G[A,B]$ be a bipartite graph with degree sequence $d=(d^A,d^B)$,
	where both $d^A$ and $d^B$ are in non-increasing order. Then $G[A,B]$ is
	decomposable if and only if there exists $p,q\in \bN$ such that
	$0<p+q<|A|+|B|$, $0\le p\le |A|$, $0\le q\le |B|$, and
	\begin{equation}\label{eq:bipindecomp}
		\sum_{i=1}^p d^A_i= p(|B|-q)+\sum_{|B|-q+1}^{|B|} d^B_i.
	\end{equation}
\end{lemma}

\begin{theorem}[\cite{erdos_new_2018}, adapted from Corollaries 6 and 9
	in~\cite{tyshkevich_decomposition_2000}]\label{thm:biptysh}\mbox{}
	\begin{enumerate}
		\item[\emph{(i)}] Any splitted bipartite degree sequence $d$ can
			be uniquely decomposed in the form \[ d=\alpha_1\circ
			\cdots\circ \alpha_k,\] where $\alpha_i$ is an
			indecomposable splitted bipartite degree sequence for
			$i=1,\ldots,k$.
		\item[\emph{(ii)}] Any realization $G$ of $d$ can be represented
			in the form \[ G=G[A_1,B_1]\circ \cdots \circ
			G[A_k,B_k],\] where $G[A_i,B_i]$ is a realization of
			$\alpha_i$.
		\item[\emph{(iii)}] Any valid bipartite switch of $G$ is a valid
			bipartite switch of $G[A_i,B_i]$ for some $i$.
	\end{enumerate}
\end{theorem}

It follows from the previous theorem that indecomposability is determined by the
degree sequence. Lemma~\ref{lem:bipindecomp} gives an explicit characterization
of such splitted bipartite degree sequences.

\begin{definition}
	Let $\overline{\cD^\circ}$ be the closure of $\cD$ under the
	composition operator $\circ$.
\end{definition}

The following theorem is a due to~Erdős, Miklós, and Toroczkai.

\begin{theorem}[Theorem~3.6 in~\cite{erdos_new_2018}]\label{thm:rapidcomp}
	If $\cD$ is rapidly mixing, then so is
	$\overline{\cD^\circ}$.
\end{theorem}

Theorem~\ref{thm:rapidcomp} is a simple consequence
of~\cite[Theorem~5.1]{erdos_decomposition_2015}. By Theorem~\ref{thm:rapidcomp},
for a class of degree sequences $\cD$ to be rapidly mixing it is sufficient that
$\operatorname{indecomp}(\cD)$ is rapidly mixing, where
\begin{equation*}
	\operatorname{indecomp}(\cD):=\{ \alpha\ |\ \alpha\text{\ is an
	indecomposable component of some }d\in\cD \}.
\end{equation*}

Because the number of realizations is independent of the internal order of the
bipartition, we revert to using ``bipartite degree sequence'' instead of the
cumbersome ``splitted bipartite degree sequence''. From now on, bipartite graphs
and their degree sequences are assumed to be splitted.

\subsection{Non-stability of Tyshkevich-compositions}\label{sec:compnonstable}

As promised, we now revisit the two examples in Section~\ref{sec:examples}.
Observe, that
\begin{align*}
	\bm{h}_0(n)&=\stackrel{n}{\overbrace{(1;1)\circ\ldots\circ (1;1)}}\\
	H_0(n)&=\stackrel{n}{\overbrace{K_2\circ\ldots\circ{K_2}}}
\end{align*}
Note, that $(1;1)=(0;\emptyset)\circ (\emptyset;0)$, so the indecomposable
decomposition of $\bm{h}_0(n)$ has $2n$ components.
Theorem~\ref{thm:biptysh} implies that $H_0(n)$ is the only realization of
$\bm{h}_0(n)$. This innocent looking example leads to the following result:

\begin{lemma}\label{lem:compnonstable}
	For any class $\cD$ of bipartite degree sequences,
	$\overline{\cD^\circ}$ is not $P$-stable (except if
	$\alpha^A=\emptyset$ for all $\alpha\in\cD$ or
	$\beta^B=\emptyset$ for all $\beta\in\cD$).
\end{lemma}
\begin{proof}
	Take $\alpha,\beta\in \cD$ such that $\alpha^A\neq\emptyset$ and
	$\beta^B\neq\emptyset$. Let
	\begin{equation*}
		\bm{d}(r)=\stackrel{r}{\overbrace{(\alpha\circ\beta)\circ\ldots\circ
		(\alpha\circ\beta)}}.
	\end{equation*}
	From Theorem~\ref{thm:biptysh} it follows that
	\begin{equation*}
		|\cG(\bm{d}(r))|={|\cG(\bm{\alpha})|}^r\cdot
	{|\cG(\bm{\beta})|}^r.
	\end{equation*}
	Let $G=(G_1\circ G_2)\circ \ldots\circ (G_{2r-1}\circ G_{2r})$ be an
	arbitrary realization of $\bm{d}(r)$ where $G_{2i-1}\in \cG(\alpha)$ and
	$G_{2i}\in \cG(\beta)$. Recall that
	$\bm{h}_0(r)-\mathds{1}_{a_1}-\mathds{1}_{b_r}$ has exponentially many
	realizations (Corollary~\ref{cor:orderofmagnitude}).

	Choose a vertex $a_i$ from the first class of $G_{2i-1}$
	and $b_{i}$ from the second class of $G_{2i}$ (for $i\in [1,r]$).
	Observe, that $G[\{a_1,\ldots,a_r\},\{b_1,\ldots,b_r\}]$ is an induced
	copy $H_0(r)$. By replacing this subgraph with a realization of
	$\bm{h}_0(r)-\mathds{1}_{a_1}-\mathds{1}_{b_r}$, an exponential number
	of realizations of $\bm{d}(r)-\mathds{1}_{a_1}-\mathds{1}_{b_r}$ are
	obtained; however, because the substitution does not change the
	components $G_{2i-1}$ and $G_{2i}$ for any $i$, $G$ is recoverable from
	such realizations. In other words, every realization of some
	$\bm{d}'\in\bS_2(\bm{d}(r))$ is obtained from at most one realization of
	$\bm{d}(r)$, so $\cD$ cannot be $P$-stable.
\end{proof}

The degree sequence $\bm{g}(n)$ was obtained by replacing $a_{i}b_{i}$ with two
independent edges. Therefore Lemma~\ref{lem:compnonstable}
applies to $\{\bm{g}(n)\ |\ n\in\bN\}$:
\begin{align*}
	\bm{g}(n)&=\stackrel{n}{\overbrace{(1,1;1,1)\circ\ldots\circ
	(1,1;1,1)}}
\end{align*}
Naturally, $\stackrel{n}{\overbrace{2K_2\circ\ldots\circ 2K_2}}$ is a
realization of $\bm{g}(n)$ and all $2^n$ realizations of $\bm{g}(n)$ are
isomorphic to it (Theorem~\ref{thm:biptysh}).

\medskip

Theorem~\ref{thm:nonstable} is not, however, a simple consequence of
Lemma~\ref{lem:compnonstable}:
\begin{lemma}
	The bipartite degree sequence $\bm{h}_k(n)$ is indecomposable for $0<k<n$.
\end{lemma}
\begin{proof}
	Via Lemma~\ref{lem:bipindecomp}. Suppose $\bm{h}_k(n)$ is decomposable.
	Substituting into~\eqref{eq:bipindecomp}, we get
	\begin{align*}
		\binom{n+1}{2}-k-\binom{n-p+1}{2}+\max\{k-p,0\}&=\\
		= p(n-q)+\binom{q+1}{2}&-\max\{k-n+q,0\} \\
		\max\{k-p,0\}+\max\{k-n+q,0\}-k & = \binom{q-p+1}{2} \\
	\end{align*}
	A short case analysis shows that the right hand side is larger than the
	left hand side.
\end{proof}

\subsection{Strong stability and
\texorpdfstring{$H_0(\ell)$}{𝐻₀(𝓁)}}\label{sec:sstable}

Strong stability is defined by Amanatidis and
Kleer~\cite{amanatidis_rapid_2019}. In their definition, they measure how
stable a degree sequence is by measuring the maximum distance of a perturbed
realization from the closest realization.
\begin{definition}[adapted from~\cite{amanatidis_rapid_2019}]\label{def:sstable}
	A degree sequence $\bm{d}$ is \emph{distance-$\ell$ strongly stable} if
	for any realization $G'$ of a degree sequence $d'$ for which
	$\|d'-d\|_1\le 2$ there exists a realization $G$ such that
	$|E(G\triangle G')|\le \ell$. A set of degree sequences is called
	\emph{strongly stable} if there exists an $\ell$ such that every degree
	sequence in the set is distance-$\ell$ strongly stable.
\end{definition}
The distance function $|E(G\triangle G')|$ used in Definition~\ref{def:sstable}
differs from the function used in~\cite{amanatidis_rapid_2019} up-to a factor of
2. Indeed, in one step, the Jerrum-Sinclair chain changes the size of the
symmetric difference by at most 2. In the other direction, suppose $G$ minimizes
$|E(G\triangle G')|$. Take a vertex $v$ where $\bm{d}(v)=\bm{d}'(v)$: $E(G)$ and
$E(G')$ evenly contribute to the edges incident to $v$ in $G\triangle G'$. For
the two vertices where $\bm{d}$ and $\bm{d'}$ differ, there is an extra edge
from $G$ or $G'$. For this reason, if $G\triangle G'$ is not a path, then it
contains a cycle $C$ whose edges alternate between $G$ and $G'$, hence $C$ is
alternating (between edges and non-edges) in $G$ as well. However,
\begin{equation*}
	|E((G\triangle C)\triangle G')|=|E(G\triangle G')|-|E(C)|,
\end{equation*}
which contradicts the minimality of $G$. If $G\triangle G'$ is path, the
Jerrum-Sinclair chain needs at most $\lceil\frac12|E(G\triangle G')|\rceil$
steps to transform $G'$ into $G$.

The way we define strong stability immediately shows that strongly stable sets
of degree sequences are also $P$-stable with $p(n)=n^{\ell+1}$.
\begin{definition}
	We say that a bipartite graph $G[A,B]$ is \emph{covered by alternating
	cycles} if for any $x\in A$ and $y\in B$ there exists a cycle $C$ which
	traverses (covers) $xy$ and alternates between the vertex sets $A$ and $B$, and
	also alternates between edges and non-edges of $G[A,B]$.
\end{definition}

\begin{lemma}
	The following statements are equivalent for a bipartite degree
	sequence $\bm{d}$.
	\begin{enumerate}
		\item $\bm{d}$ is indecomposable;
		\item every $G\in\cG(\bm{d})$ is covered by alternating cycles;
		\item every $\bm{d}'\in\bS_2(\bm{d})$ is graphic.
	\end{enumerate}
\end{lemma}
\begin{proof}
	Suppose $\bm{d}$ is decomposable; let $G\in\cG(\bm{d})$ and say
	$G=G_1\circ G_2$. Take $x\in V(G_1)\cap A$ and $y\in V(G_2)\cap B$ from
	distinct color classes, thus $xy\in E(G)$. If $\exists G'\in
	\cG(\bm{d}+\mathds{1}_x+\mathds{1}_y)$, then take $G\triangle G'$: there
	$x$ and $y$ have one extra edge in $G'$ compared to $G$, therefore there
	is an alternating path joining $x$ to $y$ in $G$ starting on an
	non-edge, i.e., there is an alternating cycle on $xy$ in $G$. This means
	that there is a realization of $\bm{d}$ in which $xy$ is not an edge.
	Therefore $\bm{d}+\mathds{1}_x+\mathds{1}_y$ is not graphic.
	The proof is similar if $x\in V(G_1)\cap B$ and $y\in V(G_2)\cap A$
	(take the complement).

	In the other direction, suppose $\bm{d}$ is indecomposable. Let
	$G\in\cG(\bm{d})$ and $\bm{d}'\in \bS_2(\bm{d})$ be arbitrary. Suppose
	first, that $\bm{d}'=\bm{d}+\mathds{1}_x+\mathds{1}_y$ where $x\in A$
	and $y\in B$. If $xy\notin E(G)$, then $G+xy$ is a realization of
	$\bm{d}'$. If $xy\in E(G)$ and there is an alternating cycle $C$ on $xy$
	in $G$, then $G\triangle C+xy\in \cG(\bm{d}')$.

	If $xy\in E(G)$ is not contained in an alternating cycle in $G$, then
	let $A_1\subset A$ and $B_1\subset B$ be the set of vertices that are
	reachable from $x$ on an alternating path starting on a non-edge. Define
	$A_2=A\setminus A_1$ and $B_2=B\setminus B_1$. We must have $y\in B_2$,
	otherwise there is an alternating cycle on $xy$. Observe, that
	$G=G[A_1,B_1]\circ G[A_2,B_2]$, a contradiction.

	If $\bm{d}'=\bm{d}-\mathds{1}_x-\mathds{1}_y$ where $x\in A$
	and $y\in B$, take the complement to arrive in the previous case.

	Finally, we have $\bm{d}'=\bm{d}-\mathds{1}_x+\mathds{1}_y$ where
	$x,y\in A$ or $x,y\in B$. Without loss of generality, suppose that
	$x,y\in A$. Let $G\in\cG(\bm{d})$ be arbitrary. If there is an
	alternating path $P$ starting on a edge from $x$ to $y$ in $G$, then
	$G\triangle P\in \cG(\bm{d}')$. If there is no such alternating path,
	take a $z$ in $B$ such that $xz\in E(G)$. Then $yz\in E(G)$, too. As
	before, there exists an alternating cycle $C$ on $yz$ in $G$, because
	$\bm{d}$ is indecomposable. Since $C$ is an alternating cycle, $xz\notin
	E(G)$, thus $G\triangle C-xz+yz\in \cG(\bm{d}')$.
\end{proof}

\begin{lemma}
	Suppose that the minimum length of an alternating cycle covering $xy$ in $G$ is
	$2\ell+2$ and $G\in\cG(\bm{d})$. Then a graphic element of $\bS_2(\bm{d})$ is not
	distance-$(2\ell)$ strongly stable. Moreover, there is an induced copy of
	$H_0(\lceil\ell/3\rceil)$ in $G$.
\end{lemma}
\begin{proof}
	Notice that all of the conclusions are invariant on complementing $G$.
	By taking the complement of $G$, we may suppose that $xy\notin E(G)$.

	Take $\bm{d}':=\bm{d}+\mathds{1}_x+\mathds{1}_y$. For any realization
	$G'\in \cG(\bm{d}')$ we have $|E(G\triangle G')|>2\ell$, otherwise
	there is an alternating path of length at most $2\ell-1$ in $G$ which
	forms an alternating cycle of length $2\ell$ with $xy$.

	Let $C$ be an alternating cycle of length $2\ell+2$ on $xy$. Let
	$a_1:=x$ and $b_{\ell+1}:=y$. Let $a_i$ and $b_i$ be the vertices
	at distance $2i-2$ and $2i-1$ from $x$ on $C-xy$, respectively.

	Notice, that $a_i b_j\in E(G)$ if $i+1\ge j$, and $a_i b_j\notin E(G)$
	if $j\le i-2$, otherwise $C$ is not the shortest alternating cycle on
	$xy$. Let
	\begin{align*}
	A':=\{a_{3i-2}\ :\ i=1,\ldots,\lceil\ell/3\rceil\},\\
	B':=\{a_{3i-1}\ :\ i=1,\ldots,\lceil\ell/3\rceil\}.
	\end{align*}
	We have
	\begin{equation*}
		G[A',B']=(a_1,\emptyset)\circ (\emptyset,b_2)\circ
		(a_4,\emptyset)\circ\cdots\circ(\emptyset,b_k)
		\simeq H_0(\lceil\ell/3\rceil).
	\end{equation*}
\end{proof}

\section{The switch Markov chain}\label{sec:markov}

For the precise definition of Markov chains and an introduction to their theory,
the reader is referred to Durrett~\cite{durrett_probability:_2010}. To define
the unconstrained and bipartite %
switch Markov chains, it is sufficient to define their transition matrices.

\begin{definition}[unconstrained/bipartite switch Markov chain]\label{def:ucChain}
	Let $\bm{d}$ be an unconstrained or bipartite degree sequence on $n$
	vertices. The state space of the switch Markov chain $\cM(\bm{d})$ is
	$\cG(\bm{d})$. The transition probability between two different states
	of the chain is nonzero if and only if the corresponding realizations
	are connected by a switch, and in this case this probability is
	$\frac{1}{6}\binom{n}{4}^{-1}$. The probability that the chain stays at
	a given state is one minus the probability of leaving the given state.
\end{definition}

It is well-known that any two realizations of an unconstrained or bipartite
degree sequence can be transformed into one-another through a series of switches.

\medskip

The switch Markov chains defined are irreducible (connected), symmetric,
reversible, and lazy. Their unique stationary distribution is the uniform
distribution $\pi\equiv |\cG(\bm{d})|^{-1}$.

\begin{definition}
	The mixing time of a Markov chain $\cM$ is
	\[
		\tau_\cM(\varepsilon)=\min\left\{ t_0\ :\ \forall x\ \forall t\ge
		t_0\ \|P^{t}(x,\cdot) -\pi \|_1\le 2\varepsilon\right\},
	\]
	where $P^{t}(x,y)$ is the probability that when $\cM$ is started
	from $x$, then the chain is in $y$ after $t$ steps.
\end{definition}

\begin{definition}
	The switch Markov chain is said to be rapidly mixing on an infinite set
	of degree sequences $\cD$ if there exists a fixed polynomial
	$\mathrm{poly}(n,\log\varepsilon^{-1})$ which bounds the mixing time of
	the switch Markov chain on $\cG(\bm{d})$ for any
	$\bm{d}\in\cD$ (where $n$ is the length of $\bm{d}$).
\end{definition}

Sinclair's seminal paper describes a combinatorial method to bound the mixing
time.

\begin{definition}[Markov graph]
	Let $G(\cM(\bm{d}))$ be the graph whose vertices are
	realizations of $\bm{d}$ and two vertices are connected by an edge if
	the switch Markov chain on $\cG(\bm{d})$ has a positive
	transition probability between the two realizations.
\end{definition}

Let $\Gamma$ be a set of paths in $\cM(\bm{d})$. We say that $\Gamma$ is a
\emph{canonical path system} if for any two realizations $G,H\in \cG(\bm{d})$ there is
a unique $\gamma_{G,H}\in \Gamma$ which joins $G$ to $H$ in the Markov graph.
The load of $\Gamma$ is defined as
\begin{equation}\label{eq:rho}
	\rho(\Gamma)=\max_{P(e)\neq 0}\frac{|\{\gamma\in \Gamma\ :\ e\in
	E(\gamma)\}|}{|\cG(\bm{d})|\cdot P(e)},
\end{equation}
where $P(e)$ is the transition probability assigned to the edge $e$ of the
Markov graph (this is well-defined because the studied Markov chains are
symmetric). The next lemma follows from Proposition 1 and Corollary 4 of
Sinclair~\cite{sinclair_improved_1992}.

\begin{lemma}\label{lem:sinclair}
	If $\Gamma$ is a canonical path system for $\cM(\bm{d})$ then
	\[
		\tau_{\cM(\bm{d})}(\varepsilon)\le \rho(\Gamma)\cdot \ell(\Gamma)\cdot
		\left(\log(|\cG(\bm{d})|) +\log(\varepsilon^{-1})
		\right),
	\]
	where $\ell(\Gamma)$ is the length of the longest path in $\Gamma$.
\end{lemma}

Obviously, $\log(|\cG(\bm{d})|)\le n^2$, henceforth we focus on bounding
$\rho$ by a polynomial of $n$.

\section{Flow representation}\label{sec:flow}

In this section we introduce a flow representation of realizations of bipartite
degree sequences defined on $A_n$ and $B_n$ as their first and second color
classes, respectively.

Let $F=F_n=(A_n,B_n,\vec E)$ be a directed bipartite graph such that
\begin{itemize}
	\item ${a_i b_j}\in \vec E(F)$ if and only if $i\le j$,
	\item ${b_j a_i}\in \vec E(F)$ if and only if $j<i$.
\end{itemize}
Every edge in $uv\in \vec E(F)$ has capacity 1 in the direction from $u$ to $v$.
We will only consider integer flows, so any admissible flow in $F$ is a subgraph
of $F$. If the sum of the flow injected at the sources is $r\in \bN$, then the
flow is called an $r$-flow.

\begin{definition}
	The flow representation $\vec\nabla(G)$ of a bipartite graph
	$G[A_n,B_n]$ is the subgraph of $F_n$ obtained as follows: take the
	symmetric difference $\nabla(G)=G[A_n,B_n]\triangle H_0(n)$, then make
	$\nabla(G)$ directed such that each edge in $\nabla(G)$ matches its
	orientation in $F_n$.
\end{definition}

\begin{lemma}\label{lem:equivalence}
	The flow representation ${\vec\nabla(G)}$ is an admissible flow
	in $F$. Moreover,
	\begin{itemize}
		\item every $a_i\in A$ is a \textbf{source} of
			${(\deg_{G}(a_i)-\deg_{H_0(n)}(a_i))}^-$ commodity,\\
			every $b_i\in B$ is a \textbf{source} of
			${(\deg_{G}(b_i)-\deg_{H_0(n)}(b_i))}^+$ commodity,
		\item every $a_i\in A$ is a \textbf{sink} of
			${(\deg_{G}(a_i)-\deg_{H_0(n)}(a_i))}^+$ commodity,\\
			every $b_i\in B$ is a \textbf{sink} of
			${(\deg_{G}(b_i)-\deg_{H_0(n)}(b_i))}^-$ commodity.
	\end{itemize}
	Conversely, such a flow is the flow representation of some
	$G[A_n,B_n]$.
\end{lemma}
\begin{proof}
	Observe the structure of $H_0(n)$ on Figure~\ref{fig:halfgraph}.
	We have
	\begin{align*}
		\deg_G(a_i)-d_{H_0(n)}(a_i)&=
		\deg_{\nabla(G)}(a_i,\{b_{1},\ldots,b_{i-1}\})-
		\deg_{\nabla(G)}(a_i,\{b_{i},\ldots,b_{n}\})\\
		&=\varrho_{\vec\nabla(G)}(a_i)-\delta_{\vec\nabla(G)}(a_i),\\
		\deg_G(b_i)-d_{H_0(n)}(b_i)&=
		\deg_{\nabla(G)}(b_i,\{a_{i+1},\ldots,a_{n}\})-
		\deg_{\nabla(G)}(b_i,\{a_{1},\ldots,a_{i}\})\\
		&=\delta_{\vec\nabla(G)}(a_i)-\varrho_{\vec\nabla(G)}(a_i).
	\end{align*}
	In the other direction, remove the orientation from the flow and take
	its symmetric difference with $H_0(n)$ to obtain the appropriate
	$G[A_n,B_n]$.
\end{proof}
\begin{corollary}\label{cor:equivalence}
	For any $\bm{d}\in \bS_{2k}(\cH_0)$, the function $G\mapsto
	\vec\nabla(G)$ is a bijection between $\cG(\bm{d})$ and such
	$k$-flows on $F$ where the sources and sinks
	prescribed according to Lemma~\ref{lem:equivalence}.
\end{corollary}
For example: every flow representation of a realization of
\begin{equation*}
	\bm{h}_0(n)-\mathds{1}_{a_1}+2\cdot\mathds{1}_{b_2}+\mathds{1}_{a_7}
	-2\cdot\mathds{1}_{b_8}
\end{equation*}
is a $3$-flow with sources at $a_1$ and $b_2$, and sinks as at $a_7$ and $b_8$;
see Figure~\ref{fig:flow}.

\begin{figure}[H]
\centering
\begin{tikzpicture}[scale=1.7]
\tikzstyle{vertex}=[draw,circle,fill=black,minimum size=3,inner sep=0]

\foreach \j in {1,...,8}
{%
	\node[vertex] (y\j) at (\j,0) [label=south:{$a_{\j}$}] {};
	\node[vertex] (x\j) at (\j+0.5,1) [label=north:{$b_{\j}$}] {};
}

\draw (y1) edge[->] (x3)(x3) edge[->] (y4)(y4) edge[->] (x5)
	(x5) edge[->] (y8)(y8) edge[->] (x8)
	(x2) edge[->] (y4)(y4) edge[->] (x6)
	(x6) edge[->] (y7)(y7) edge[->] (x8)
	(x2) edge[->] (y3)(y3) edge[->] (x5)
	(x5) edge[->] (y7);
\end{tikzpicture}
\caption{The flow representation of a realization of a degree sequence from
$B_{6}(h_0(8))$.}%
\label{fig:flow}
\end{figure}
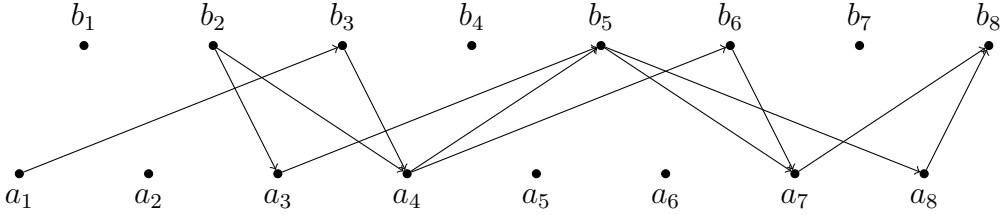

\section{Proof of Theorem~\ref{thm:main}: rapid mixing on
\texorpdfstring{$B_{2k}(\cH_0)$}{the 2𝑘-ball around 𝓗₀}}\label{sec:main}

\subsection{Overview of the proof}

Without loss of generality $\bm{d}\in \bS_{2k}(\cH_0)$. Let $X,Y \in
\cG(\bm{d})$ be two distinct realizations. We will define a switch sequence
\begin{equation*}
	\gamma_{X,Y}: X=Z_0,Z_1,\ldots, Z_{t}=Y,
\end{equation*}
We will also define a set of corresponding encodings
\begin{equation*}
	L_0(X,Y),L_1(X,Y), \ldots,L_{t}(X,Y).
\end{equation*}
The canonical path system $\Gamma:=\{\gamma_{X,Y}\ |\ X,Y\in
\cG(\bm{d})\}$ on $G(\cM(\bm{d}))$ will satisfy the
following two properties:
\begin{itemize}
	\item \textbf{Reconstructible:} there is an algorithm that for each
		$i$, takes $Z_i$ and $L_i(X,Y)$ as an input and outputs the
		realizations $X$ and $Y$.

	\item \textbf{Encodable in $\cG(\bm{d})$:} the total number of
		encodings on each vertex of $G(\cM(\bm{d}))$ is at
		most a $\mathrm{poly}_k(n)$ factor larger than $|\cG(\bm{d})|$.
\end{itemize}
The ``Reconstructible'' property ensures that the number of canonical paths
traversing a vertex (and thus an edge) of the Markov graph $\cM(\bm{d})$ is at
most the size of the set of all possible encodings. Subsequently, by
substituting into Equation~\eqref{eq:rho}, the ``Encodable in $\cG(\bm{d})$''
property implies that $\rho(\Gamma)=\cO(\mathrm{poly}_k(n))$.
According to Lemma~\ref{lem:sinclair}, the last bound means that the bipartite
switch Markov chain is rapidly mixing.

\medskip

Now we give a description of how the $X=Z_0,Z_1, \ldots, Z_{t+1}=Y$ canonical
path is constructed. The main idea is to morph $X$ into $Y$ ``from left to
right'': a region of constant width called the \textbf{buffer} is moved
peristaltically through $A_n\cup B_n$, consuming $X$ on its right and producing
$Y$ on its left; see Figure~\ref{fig:buffer}.

The encoding $L_i$ will contain a realization whose structure is similar to
$Z_i$, but the roles of $X$ and $Y$ are reversed. Furthermore, $L_i$ will contain
the position of the buffer and some additional information about the vertices in
the buffer.

\begin{figure}[H]
\centering
\begin{tikzpicture}[decoration={brace,mirror,amplitude=7}]

\node () at (3,3.5){\large The structure of a typical intermediate realization $Z_i$};
\node (buffer) at (3,2){\Large Buffer};

\draw(1,1)--(5,1)--(5,3)--(1,3)--(1,1);
\draw(5,2.5)--(10,2.5);
\draw(1,2.5)--(-4,2.5);
\draw(5,1.5)--(10,1.5);
\draw(1,1.5)--(-4,1.5);

\node () at (7.5,2){End of $X$};
\node () at (-1.5,2){Beginning of $Y$};

\node () at (-3.7,2.8){$b_1$};
\node () at (-3.2,2.8){$b_2$};
\node () at (-1.5,2.8){$\cdots$};
\node () at (0.5,2.8){$b_i$};
\node () at (9.7,2.8){$b_n$};
\node () at (9,2.8){$b_{n-1}$};
\node () at (7.5,2.8){$\cdots$};
\node () at (5.7,2.8){$b_{i+z+1}$};

\node () at (-3.7,1.2){$a_1$};
\node () at (-3.2,1.2){$a_2$};
\node () at (-1.5,1.2){$\cdots$};
\node () at (0.5,1.2){$a_i$};
\node () at (9.7,1.2){$a_{n}$};
\node () at (9,1.2){$a_{n-1}$};
\node () at (7.5,1.2){$\cdots$};
\node () at (5.7,1.2){$a_{i+z+1}$};

\draw [decorate] ([yshift=-5mm]1,1.3) --node[below=3mm]{Constant width} ([yshift=-5mm]5,1.3);

\begin{scope}[shift={(0,-5)}]%

\node () at (3,3.5){\large The realization in the corresponding $L_i$};
\node (buffer) at (3,2){\Large Buffer};

\draw(1,1)--(5,1)--(5,3)--(1,3)--(1,1);
\draw(5,2.5)--(10,2.5);
\draw(1,2.5)--(-4,2.5);
\draw(5,1.5)--(10,1.5);
\draw(1,1.5)--(-4,1.5);

\node () at (7.5,2){End of $Y$};
\node () at (-1.5,2){Beginning of $X$};

\node () at (-3.7,2.8){$b_1$};
\node () at (-3.2,2.8){$b_2$};
\node () at (-1.5,2.8){$\cdots$};
\node () at (0.5,2.8){$b_i$};
\node () at (9.7,2.8){$b_n$};
\node () at (9,2.8){$b_{n-1}$};
\node () at (7.5,2.8){$\cdots$};
\node () at (5.7,2.8){$b_{i+z+1}$};

\node () at (-3.7,1.2){$a_1$};
\node () at (-3.2,1.2){$a_2$};
\node () at (-1.5,1.2){$\cdots$};
\node () at (0.5,1.2){$a_i$};
\node () at (9.7,1.2){$a_{n}$};
\node () at (9,1.2){$a_{n-1}$};
\node () at (7.5,1.2){$\cdots$};
\node () at (5.7,1.2){$a_{i+z+1}$};

\draw [decorate] ([yshift=-5mm]1,1.3) --node[below=3mm]{Constant width}
	([yshift=-5mm]5,1.3);

\end{scope}

\end{tikzpicture}
\caption{A realization along $\gamma_{X,Y}$ and the main part of the associated
encoding.}%
\label{fig:buffer}
\end{figure}

Let $\overline{A}_i=A_n\setminus A_i$ and $\overline{B}_i=B_n\setminus B_i$.
Also, let $U_i=A_i\cup B_i$ and $\overline{U}_i=\overline{A}_i\cup
\overline{B}_i$. The following lemma shows the existence of a suitable buffer
which can be used to interface two different realizations as displayed on
Figure~\ref{fig:buffer}.

\begin{lemma}\label{lem:buffers}
	If $i,z\in \bN$ satisfy $0 \leq i \leq n-z$ and
	$2k+\sqrt{2k}+1\leq z$, then there is a realization
	$T_{X,Y}[i+1,i+z]\in \cG(\bm{d})$ with the following properties:
	\begin{itemize}
		\item $U_{i}$ induces identical subgraphs in $T_{X,Y}[i+1,i+z]$
			and $Y$, and
		\item $\overline{U}_{i+z}$ induces identical subgraphs in
			$T_{X,Y}[i+1,i+z]$ and $X$.
	\end{itemize}
	For $k=1$, even $z=1$ is sufficient.
\end{lemma}
\begin{proof}
	We will work with the flow representation of $X$ and $Y$. Since $X$ and
	$Y$ are the realizations of the same degree sequence, the source-sink
	distribution in their corresponding flow representation is identical.
	It is sufficient to design a flow which joins the flow $\vec\nabla(X)$
	leaving $U_i$ and redirects it to the vertices in $\overline{U}_{i+z}$
	with the same distribution as $\vec\nabla(Y)$ flows into them from
	$U_{i+z}$.

	The case $k=z=1$ can be manually checked at this point.

	To achieve the outlined goal for any $k$, we define an auxiliary network
	$F'$ and prescribe the flow corresponding to the buffer on it. Let
	$e_D(W,Z)$ be the number of edges of $D$ that are directed from $W$ to
	$Z$.
	\begin{align*}
		A_X&:=\{ a_j\in A_i\ |\ e_{\vec\nabla(X)}(a_j,
		\overline{B}_i)>0\}\\
		B_X&:=\{ b_j\in B_i\ |\ e_{\vec\nabla(X)}(b_j,
		\overline{A}_i)>0\}\\
		A_Y&:=\{ a_j\in \overline{A}_{i+z}\ |\
		e_{\vec\nabla(Y)}(B_{i+z},a_j)>0\}\\
		B_Y&:=\{ b_j\in \overline{B}_{i+z}\ |\
		e_{\vec\nabla(Y)}(A_{i+z},b_j)>0\}\\
		A'&:=A_X\cup(A_{i+z}\setminus A_i)\cup A_Y\\
		B'&:=B_X\cup(B_{i+z}\setminus B_i)\cup B_Y
	\end{align*}
	The underlying network $F'$ is a subgraph of $F$:
	\begin{align*}
		F'&:=F[A',B']-E(F[A_X\cup A_Y,B_X\cup B_Y]),
	\end{align*}
	i.e., the flow cannot use edges between $A_X$, $B_X$, $A_Y$, $B_Y$.
	Note, that to prove the lemma for $k=z=1$, one has to use edges of
	$F[A_X,B_Y]$ and $F[A_Y,B_X]$.

	The flow in the buffer will be a subgraph $W\subset F'$.
	Let us define $f:A'\cup B'\to \bZ$:
	\begin{align*}
		f(a_j)&:=\left\{
		\begin{array}{ll}
			e_{\vec\nabla(X)}(a_j,\overline{B}_i), & \text{ if }
			a_j\in A_X\\
			-e_{\vec\nabla(Y)}(B_{i+z},a_j), & \text{ if } a_j\in
			A_Y\\
			\deg_{H_0(n)}(a_j)-\deg_{\bm{d}}(a_j), & \text{ if
			}a_j\in A_{i+z}\setminus A_i
		\end{array}
		\right.\\
		f(b_j)&:=\left\{
		\begin{array}{ll}
			e_{\vec\nabla(X)}(b_j,\overline{A}_i), & \text{ if }
			b_j\in B_X\\
			-e_{\vec\nabla(Y)}(A_{i+z},b_j), & \text{ if } b_j\in
			B_Y\\
			\deg_{\bm{d}}(b_j)-\deg_{H_0(n)}(b_j), & \text{ if
			}b_j\in B_{i+z}\setminus B_i
		\end{array}
		\right.
	\end{align*}
	We prescribe sources and sinks in $W$ as follows (recall
	Lemma~\ref{lem:equivalence}):
	\begin{align*}
		\delta_W(a_j)-\varrho_W(a_j)&=f(a_j)\quad\forall a_j\in A'\\
		\delta_W(b_j)-\varrho_W(a_j)&=f(b_j)\quad\forall b_j\in B'
	\end{align*}
	If such a $W$ exists, then $\vec\nabla(X)[A_i,B_i]+W+
	\vec\nabla(Y)[\overline{A}_{i+z},\overline{B}_{i+z}]$ is a $k$-flow
	which, according to Corollary~\ref{cor:equivalence}, corresponds to a
	graph whose degree sequence is $\bm{d}$.

	The existence of $W$ is proved using Menger's theorem on the number of
	edge-disjoint directed $st$-paths. It is sufficient to show that any
	$S\subseteq A'\cup B'$ satisfies the cut-condition:
	\begin{equation}\label{eq:cut}
		\delta_{F'}(S)\ge \sum_{s\in S}{f(s)}
	\end{equation}
	Trivially, the right-hand side is at most $k$. Let us take an $S$ for
	which $\delta_{F'}(S)-\sum_{s\in S}{f(s)}$ is minimal. We claim that
	the following four statements hold:
	\begin{itemize}
		\item If $|S\cap (A_{i+z}\setminus A_i)|> k$, then $B_Y\subset
			S$.
		\item If $|S\cap (B_{i+z}\setminus B_i)|> k$, then $A_Y\subset
			S$.
		\item If $|S\cap (A_{i+z}\setminus A_i)|< z-k$, then $B_X\cap
			S=\emptyset$.
		\item If $|S\cap (B_{i+z}\setminus B_i)|< z-k$, then $A_X\cap
			S=\emptyset$.
	\end{itemize}
	We only prove the first statement because the rest can be shown
	analogously. Suppose $|S\cap (A_{i+z}\setminus A_i)|>k$ and $b_j\in
	B_Y$, but $b_j\notin S$. Moving $b_j$ into $S$ changes the difference
	between the two sides of~\eqref{eq:cut} by
	\begin{equation*}
		-|S\cap (A_{i+z}\setminus A_i)|-f(b_j)<
		-k+e_{\vec\nabla(Y)}(A_{i+z},b_j)\le 0,
	\end{equation*}
	which contradicts the minimality of $S$.

	Finally, we have four cases. In each case we show that~\eqref{eq:cut} holds.
	\begin{itemize}
		\item \textbf{Case 1}: $|S\cap (A_{i+z}\setminus A_i)|\le k$
			and $|S\cap (B_{i+z}\setminus B_i)|\ge z-k$. We have
			\begin{equation*}
				\delta_{F'}(S)\ge e_{F'}\left(S\cap (B_{i+z}
				\setminus B_i),(A_{i+z}\setminus A_i)\setminus
				S\right)\ge \sum_{r=1}^{z-2k-1}r
				\ge \binom{z-2k}{2}\ge k,
			\end{equation*}
			thus $S$ satisfies~\eqref{eq:cut}.
		\item \textbf{Case 2}: $|S\cap (A_{i+z}\setminus A_i)|\le k$
			and $|S\cap (B_{i+z}\setminus B_i)|\ge z-k$: as in Case
			1, we get that $\delta_{F'}(S)\ge k$.
		\item \textbf{Case 3}: $|S\cap (A_{i+z}\setminus A_i)|> k$ and $|S\cap
			(B_{i+z}\setminus B_i)|> k$. By our previous
			statements, we have $A_Y\cup B_Y\subseteq S$. Consequently,
			\begin{align*}
			\delta_{F'}(S)&\ge\delta_{\vec\nabla(X)\cap F'}(S)=
			\delta_{\vec\nabla(X)}(S\cup\overline{U}_{i+z})-
			\delta_{\vec\nabla(X)\cap F[A_i,B_i]}(S)=\\
			&=\sum_{s\in S\cup \overline{U}_{i+z}}
			\left(\delta_{\vec\nabla(X)}(s)-\varrho_{\vec\nabla(X)}(s)\right)
			-\delta_{\vec\nabla(X)\cap F[A_i,B_i]}(S)=\\
			&=-\sum_{s\in
			\overline{U}_{i+z}}e_{\vec\nabla(X)}(U_{i+z},s)+\sum_{s\in
			S\cap ({U}_{i+z}\setminus U_i)}f(s)
			+\sum_{s\in U_i}e_{\vec\nabla(X)}(s,\overline{U}_i)=\\
			&=-\sum_{s\in
			\overline{U}_{i+z}}e_{\vec\nabla(Y)}(U_{i+z},s)+\sum_{s\in
			S\cap {U}_{i+z}}f(s)=\sum_{s\in S}f(s),
			\end{align*}
			which is what we wanted to show.
		\item \textbf{Case 4}: $|S\cap (A_{i+z}\setminus A_i)|< z-k$
			and $|S\cap (B_{i+z}\setminus B_i)|<z-k$: by our
			previous statements, we have $S\cap (A_X\cup
			B_X)=\emptyset$. Since
			$\delta_{F'}(S)=\varrho_{F'}(A'\cup B'\setminus S)$,
			the proof is practically the same as that of Case 3, we
			can use $\vec\nabla(Y)$ to demonstrate
			that~\eqref{eq:cut} is satisfied by $S$.
	\end{itemize}
\end{proof}

\subsection{Constructing the canonical path
\texorpdfstring{$\gamma_{X,Y}$}{𝛾\_\{𝑋,𝑌\}}.}\label{subseq:pathdef}

We will explicitly construct $2(n-3k-3)+1$ intermediate realizations along the
switch sequence $\gamma_{X,Y}$. Let $X$ and $Y$ be the two different
realizations which we intend to connect. The switch sequence includes
$T_{X,Y}[i+1,i+3k+1]$, $T_{X,Y}[i+1,i+3k+2]$, $T_{X,Y}[i+2,i+3k+2]$ for each
$i=1,\ldots,n-3k-3$ in increasing order. These realizations are called
\textbf{milestones}. A roadmap is shown on Figure~\ref{fig:plan}.

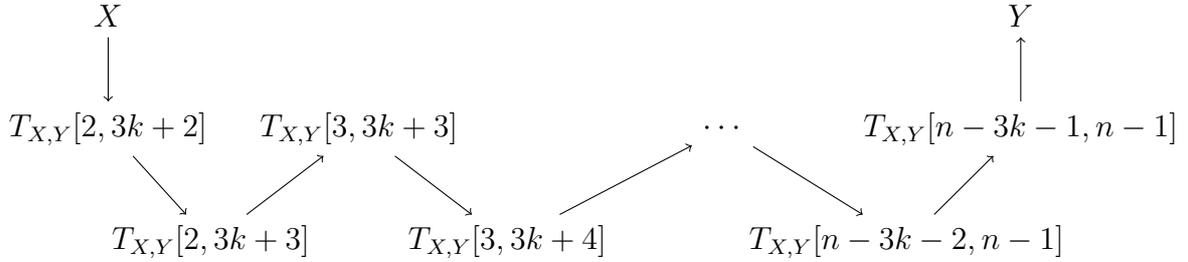
\begin{figure}[H]
\centering
\begin{tikzpicture}[scale=1.5]
\tikzstyle{vertex}=[draw,circle,fill=black,minimum size=3,inner sep=0]

\node (x) at (0,2) {$X$};
\node (x2) at (0,1) {$T_{X,Y}[2,3k+2]$};
\node (x3) at (0.9,0) {$T_{X,Y}[2,3k+3]$};
\node (x4) at (2.2,1) {$T_{X,Y}[3,3k+3]$};
\node (x5) at (3.5,0) {$T_{X,Y}[3,3k+4]$};

\node (dots) at (5.4,1) {$\cdots$};

\node (x6) at (7,0) {$T_{X,Y}[n-3k-2,n-1]$};
\node (x7) at (8,1) {$T_{X,Y}[n-3k-1,n-1]$};
\node (y) at (8,2) {$Y$};

\draw[->] (x) edge (x2) (x2) edge (x3)(x3) edge (x4)(x4) edge (x5)(x5)
	edge (dots)(dots) edge (x6)(x6) edge (x7)(x7) edge (y);

\end{tikzpicture}
\caption{Roadmap of the switch sequence between $X$ and $Y$. The existence of a
short switch sequence between milestones of the sequence is guaranteed by
Lemma~\ref{lem:shortsequence}.}%
\label{fig:plan}
\end{figure}

\begin{lemma}\label{lem:shortsequence}
	There is a switch sequence of length $\cO(k^2)$ that connects
	$T_{X,Y}[i+1,i+z]$ to $T_{X,Y}[i+1,i+z+1]$ and $T_{X,Y}[i+2,i+z]$ to
	$T_{X,Y}[i+1,i+z+1]$.
\end{lemma}
\begin{proof}
	According to~\cite{erdos_swap-distances_2013}, there is a switch sequence
	of length at most
	\begin{equation*}
		\frac{|E(T_{X,Y}[i+1,i+z])\triangle
		E(T_{X,Y}[i+1,i+z+1])|}{2}\le \frac12{(z+1+2k)}^2\le
		\frac12{(5k+2)}^2,
	\end{equation*}
	between $T_{X,Y}[i+1,i+z]$ and $T_{X,Y}[i+1,i+z+1]$, since they induce
	identical graphs on $U_{i}$ and $\overline{U}_{i+z+1}$, and the at most
	$k-k$ edges entering $U_{i+1}$ and leaving $U_{i+z+1}$ in the flow
	representations are incident on the same set of vertices in the two
	flows.
\end{proof}
Note that in Lemma~\ref{lem:buffers}, $X$ satisfies the role of
$T_{X,Y}[1,3k+2]$ and $Y$ satisfies the role of $T_{X,Y}[n-3k-1,n]$. By
applying Lemma~\ref{lem:shortsequence}, the arrows in Figure~\ref{fig:plan} can
be substituted with switch sequences of constant length. Concatenating these
short switch sequences and pruning the circuits from the resulting trail (so that
any realization is visited at most once by the canonical path) produces the
switch sequence $\gamma_{X,Y}$ connecting $X$ to $Y$ in the Markov graph.

\subsection{Assigning the encodings.}

Each realization visited by $\gamma_{X,Y}$ receives an encoding that will
be an ordered 4-tuple consisting of another realization, two graphs of
constant size, and an integer in $\{1, \ldots, n\}$.

The closed neighborhood of a subset of vertices $U\subseteq V(G)$ in a
graph $G$ is denoted by $N_G[U]\supseteq U$. For the two graphs of
constant size, we need the following definition.

\begin{definition}[left-compressed induced subgraph]
	Let $X$ be a realization and let $R\subset A\cup B$. Let us group the
	vertices $A\cup B$ into pairs: ${\{(a_i, b_{i})\}}_{i=1}^{n}$. Remove
	the pairs that do not intersect $R$, and let the remaining pairs be
	${\{(a_{i_j}, b_{i_j})\}}_{j=1}^{r}$ for some $i_1<\cdots<i_r$. For
	each edge of $E(X[R])$, map $a_{i_j}\mapsto a_j$ and $b_{i_j}\mapsto
	b_{j}$ for all $j$ simultaneously. This changes the embedding of the
	vertices of $X[R]$, and we call this new graph the left-compressed copy
	of $X[R]$.
\end{definition}

To any realization on the canonical path $\gamma_{X,Y}$ we will assign an
encoding
\begin{equation*}
	L_i(X,Y):=\Big(T_{Y,X}[i+1,i+3k+1],G_X[i],G_Y[i],i\Big)
\end{equation*}
for some $0\le i\le n-3k-1$, where $G_X[i]$ is the left-compressed subgraph of
$X$ induced by $N_{\vec\nabla(X)}[U_{i+3k+1}\setminus U_{i}]$ and $G_Y[i]$ is
the left-compressed copy of subgraph of $Y$ induced
$N_{\vec\nabla(Y)}[U_{i+3k+1}\setminus U_i]$. An encoding is assigned to each
realization along the switch sequence $\gamma_{X,Y}$ as follows:
\begin{itemize}
	\item The encoding $L_0(X,Y)$ (where $T_{Y,X}[1,3k+1]:=Y$) is used from
		the beginning $X$ of the switch sequence until it arrives at
		$T_{X,Y}[2,3k+3]$ (including this realization).

	\item For $1\le i \le n-3k-1$, the encoding $L_i$ is used between
		$T_{X,Y}[i+1,i+3k+2]$ (not included) and $T_{X,Y}[i+2,i+3k+2]$
		(included).

	\item The encoding $L_{n-3k-1}$ (where $T_{Y,X}[n-3k,n]:=X$ is chosen)
		is used from $T_{X,Y}[n-3k-2,n-1]$ (not included) to $Y$.
\end{itemize}

\subsection{Estimating the load
\texorpdfstring{$\rho(\Gamma)$}{𝜌{(𝛤)}}}\label{sec:load}

The total number of possible encodings is at most
\[
	\cO_k(|\cG(\bm{d})|\cdot n)
\]
(where the index $k$ warns that this expression may depend on $k$), since the
number of left-compressed graphs on at most $5k+2$ vertices is a constant
depending only on $k$.

\begin{lemma}[Reconstructability]
	Given $\bm{d}$, there is an algorithm that takes $Z_i\in\gamma_{X,Y}$
	and $L_i(X,Y)$ as an input and outputs the realizations $X$ and $Y$ (for any $i$).
\end{lemma}
\begin{proof}
	The first coordinate of $L_i$ is an realization, of the form
	$T_{Y,X}[i+1,i+3k+1]$ for an unknown $X,Y$. The index $i$ is known,
	because it is the last component of $L_i$. W.l.o.g.\ we show how to
	recover $X$. From $T_{Y,X}[i+1,i+3k+1]$ and $i$, we know the induced
	subgraph of $X$ on the vertices $U_i$. Similarly, the induced subgraph
	of $Z_i$ on the vertices $\overline{U}_{i+3k+1}$ is identical to the
	induced subgraph of $X$ on the same vertices. Hence the only unknown
	part of $X$ is its induced subgraph on $N_X[U_{i+3k+1}\setminus U_i]$.
	The subgraph in the second component of $L_i(X,Y)$ is the
	left-compressed copy of $X[N_X[U_{i+3k+1}\setminus U_{i}]]$. Since
	left-compression preserves the order of the indices of $a_j\in A$ and
	$b_j\in B$, $X$ can be fully recovered.
\end{proof}

\begin{proof}[Proof of Theorem~\ref{thm:main}]
	We have shown that $\rho(\Gamma)=\cO(n\cdot n^4)$ and
	$\ell(\Gamma)=\cO(n)$, thus $\tau(\varepsilon)\le
	\cO(n^{8}\log\varepsilon^{-1})$, verifying that the switch Markov
	chain is rapidly mixing on $\bS_{2k}(\cH_0)$ and $B_{2k}(\cH_0)$.
\end{proof}

\section{Proof of Theorem~\ref{thm:nonstable}: non-stability of
\texorpdfstring{$\cH_k$}{𝓗ₖ}}\label{sec:nonstable}

In this section we show that it is relatively straightforward to get the
asymptotic growth rate of the number of realizations of $\bm{h}_k(n)$ when $k$
is a constant and $n$ tends to infinity. We first illustrate this for $k=1$.
Recall Corollary~\ref{cor:equivalence} and that
$\bm{h}_1(n)=\bm{h}_0(n)-\mathds{1}_{a_1}-\mathds{1}_{b_1}$.

\begin{lemma}\label{lem:counting1}
	The number of all directed paths (integer 1-flows) from $a_1$ to $b_n$ in
	$F_n$ is
	\[
		\left[\begin{array}{c}
		1 \\
		1
		\end{array}\right]^T
		\left[ \begin{array}{cc}
		2 & 1 \\
		1 & 1
		\end{array}\right]^{n-1}
		\left[ \begin{array}{c}
		0 \\
		1
		\end{array} \right].
	\]
\end{lemma}
\begin{proof}
	Let $S_{1}(\ell)$ be the number of paths in $F_n$ that start at $a_1$
	and end in $B_\ell$. Similarly, let $S_{2}(\ell)$ be the number of
	paths in $F_n$ that start at $a_1$ and end in one of the vertices in
	$A_\ell$. We have
	\begin{equation}\label{eq:recursion}
		\begin{array}{rcl}
			S_{1}(\ell+1)&=&2S_{1}(\ell)+S_{2}(\ell),\\
			S_{2}(\ell+1)&=&S_{1}(\ell)+S_{2}(\ell).
		\end{array}
	\end{equation}
	Observe that $a_1\to b_n$ paths in $F_n$ are in bijection with paths
	starting at $a_1$ and ending in $A_n$: the corresponding paths are
	obtained by deleting the last edge incident to $b_n$. Since
	$S_{1}(1)=S_{2}(1)=1$, from~\eqref{eq:recursion} we get that $S_2(n)$
	is the quantity in the statement of the Lemma and the proof is
	complete.
\end{proof}

\begin{corollary}\label{cor:orderofmagnitude}
	The number of realizations of $\bm{h}_1(n)$ is $\Theta \left( {\left(
	\frac{3+\sqrt{5}}{2}\right)}^n \right)$
\end{corollary}
\begin{proof}
	Neither $[1,1]$ nor ${[0,1]}^T$ is perpendicular to the eigenvector that
	belongs to the largest eigenvalue $\frac{3+\sqrt{5}}{2}$ of the matrix
	\[
		\left[ \begin{array}{cc}
		2 & 1 \\
		1 & 1
		\end{array}\right].
	\]
\end{proof}

The proof of Lemma~\ref{lem:counting1} can be interpreted as follows. We count
$a_1\to b_n$ paths by looking at their induced subgraphs on the vertices in
$U_{\ell}$ (the number of these is precisely $S_1(\ell)+S_{2}(\ell)$). The main
observation is that the number of ways an $a_1\to
U_{\ell}$ path can be extended to an $a_1\to U_{\ell+1}$ path only depends on
whether the path's endpoint lies in $A_\ell$ or in $B_\ell$.

Again, according to Corollary~\ref{cor:equivalence}, realizations of
$\bm{h}_k(n)$ are in a 1-to-1 correspondence with integer $k$-flows from $a_1$
to $b_n$. We shall mimic the argument of Lemma~\ref{lem:counting1} with
$k$-flows. The recursion will consider the beginning of a $k$-flow on $U_\ell$
and its ``termination-type''.
\begin{definition}[set of types]
	Let $\cP_k$ be the set of partitions of $k$ (the set of
	multisets of positive integers whose sum of elements is exactly $k$)
	and $\cP_0:= \{\emptyset\}$. For all positive integers $k$, we
	define \emph{the set of types}:
	\begin{equation*}
		\cT_k := \{(R,Q) \, | \, \, \exists \, 0\leq m \leq k :
		R \in \cP_{m} , Q \in \cP_{k-m} \}.
	\end{equation*}
\end{definition}

\begin{definition}[type of a flow]
	Let $X$ be $k$-flow in $F_n[U_\ell]$ from a single source $a_1$, and the
	sinks are arbitrarily distributed in $U_\ell$. We say that \emph{the
	type of $X$} is $T=(R,Q)\in \cT_k$ if there is an injective function $f
	: R \rightarrow A_\ell$ such that for every $a_i\in f(R)$ we have
	\begin{equation*}
		\varrho_{X}(a_i)-\delta_{X}(a_i)=f^{-1}(a_i),
	\end{equation*}
	and for all $a_i\in A_\ell\setminus f(R)$ we have
	$\varrho_{X}(a_i)=\delta_{X}(a_i)$. Similarly, there is an injective
	function $g: Q \rightarrow B_\ell$ such that
	for every $b_i\in g(Q)$
	\begin{equation*}
		\varrho_{X}(b_i)-\delta_{X}(b_i)=g^{-1}(b_i),
	\end{equation*}
	and for all $b_i\in B_\ell\setminus g(Q)$ we
	have $\varrho_{X}(b_i)=\delta_{X}(b_i)$.
\end{definition}

Informally, the type of $X$ describes the multiplicities of
the incidences of the endpoints of the $k$-flow on $U_\ell$.

In the proof of Lemma~\ref{lem:counting1}, the functions $S_1(\ell), S_2(\ell)$
were actually the number of $1$-flows on $U_{\ell}$ of type $(\emptyset,\{1\})$
and $(\{1\},\emptyset)$, respectively. The next definition is the analogue of
the matrix in the proof of Corollary~\ref{cor:orderofmagnitude} for large $k$.

\begin{definition}[type matrix]
	For all $k$, let us fix an ordering of the types: $\cT_k=(T_1,
	\ldots, T_{|\cT_k|})$. Let $\ell$ and $n$ be so large, that
	there exists a $k$-flow which has type $T_i$ on $U_{\ell}$. We define
	$p_{i,j}$ to be the number possible ways a $k$-flow on $U_{\ell}$ from
	the single source $a_1$ can be extended to a $k$-flow of type $T_j$ on
	$U_{\ell+1}$. We define the type-matrix $\cP_k$ to be the
	$|\cT_k| \times |\cT_k|$ matrix whose element in the
	$i$-th row and $j$-th column is $p_{i,j}$.
\end{definition}

It is not hard to see that $p_{i,j}$ is well-defined, in other words, $p_{i,j}$
does not depend on either $\ell$, $n$, or the $k$-flow.

In the proof of Corollary~\ref{cor:orderofmagnitude}, the type matrix
\begin{equation*}
	\cP_1= \left[ \begin{array}{cc}
	2 & 1 \\
	1 & 1
	\end{array}\right]
\end{equation*}
corresponds to the ordering $\cT_1= \big((\emptyset,\{1\}) \, , \,
(\{1\},\emptyset)\big)$. Now we are ready to prove the analogue of
Lemma~\ref{lem:counting1} for $k$-flows where $k>1$.

\begin{lemma}
	For every $k\ge 1$, the number of $k$-flows on $F_n$ from the single
	source $a_1$ to the single sink $b_n$ is
	\begin{equation*}
		v^T \cP_k^{n-1}w
	\end{equation*}
	where:
	\begin{itemize}
		\item $v$ is the vector of length $|\cT_k|$ which
			contains $1$ at the coordinates which correspond to the
			types $(\{k-1\},\{1\})$,
			$(\{k\},\emptyset)\in\cT_k$, and zero
			everywhere else,

		\item $\cP_k$ is the type-matrix,

		\item $w$ is the vector of length $|\cT_k|$ that
			contains $1$ at the coordinate that corresponds to the
			type
			\begin{equation*}
				(\stackrel{k}{\overbrace{\{1,1,\ldots,
				1\}}},\emptyset)\in \cT_k
			\end{equation*}
			and zero everywhere else.
	\end{itemize}
\end{lemma}
\begin{proof}
	With the appropriate substitutions, the proof is identical to the proof
	of Lemma~\ref{lem:counting1}. The type of a $k$-flow on $U_1$ emanating
	from $a_1$ is either $(\{1\},\{k-1\})$ or $(\emptyset,\{k\})$. By the
	definition of $\cP_k$, the vector $v^T\cP^{n-1}$ contains the number of
	graphs on the vertices $U_n$ with a given type.  Of these, the
	$k$-flows from $a_1\to b_n$ correspond to graphs with type
	$(\emptyset,\{1,1,\ldots, 1\})$ (deleting $b_n$ and the incident edges
	results in a $k$-flow of this type). Hence the statement of the lemma
	follows.
\end{proof}

The following simple property of the type matrix will be used.

\begin{definition}
	A matrix $\cP$ is primitive, if $\exists m$ for which every
	entry of $\cP^{m}$ is positive.
\end{definition}

\begin{lemma}\label{lem:primitive}
	The type matrix $\cP_k$ is primitive for any $k$.
\end{lemma}
\begin{proof}
	For every type $t \in \cT_k$ it is easy to design a $k$-flow
	$X$ such that the type of $X[U_{\ell}]$ (for some $\ell$) is $t$ and
	the type of $X[U_{\ell+k}]$ is $(\{1,1,\ldots, 1\},\emptyset)$. Hence
	in $\cP_k^k$ the row and column that correspond to the type
	$(\{1,1,\ldots, 1\},\emptyset)$ are strictly positive. Since
	$\cP_k$ is non-negative, it also follows that
	$\cP_k^{2k}$ is positive.
\end{proof}

Now we are ready to prove the key lemma to refute the $P$-stability of the
class of degree sequences $\cH_k$.

\begin{lemma}\label{lem:eigenval}
	For every $k$, the largest eigenvalue of the type-matrix
	$\cP_k$ is smaller than the largest eigenvalue of the type
	matrix $\cP_{k+1}$.
\end{lemma}
\begin{proof}
	By Lemma~\ref{lem:primitive}, both $\cP_k$ and $\cP_{k+1}$
	are primitive. By the Perron-Frobenius theory, they both have a real
	positive eigenvalue $r_k$ and $r_{k+1}$, respectively, that is larger
	in absolute value than all of their other eigenvalues. Moreover, both
	limits
	\[
		\lim_{n \rightarrow \infty}\frac{\cP_k^n}{r_k^n} \quad
		\text{ and } \quad \lim_{n \rightarrow
		\infty}\frac{\cP_{k+1}^n}{r_{k+1}^n}
	\]
	exist and are one dimensional projections. Let the set of types $S
	\subset \cT_{k+1}$ be defined as follows:
	\[
		S:= \{ (R,Q) \in \cT_{k+1} \, : \, 1 \in R \}.
	\]
	Let $M^{(n)}$ be the principal minor of $\cP_{k+1}^n$ that is
	obtained by taking those rows and columns which correspond to types in
	$S$. Without loss of generality, we may assume that if the $i$-th row
	of $M^{(1)}$ corresponds to a type $(R,Q)$, then the $i$-th row of
	$\cP_k$ corresponds to the type $(R\setminus \{1\},Q)$.
	Moreover, we may assume that the ordering of $\cT_k$ and
	$\cT_{k+1}$ is compatible in the following sense: if
	$T=\{R,Q\}$ and $T'=(R',Q')$ are types in $S$ and $T < T'$ according to
	the ordering on $\cT_{k+1}$, then $(R\setminus\{1\},Q) <
	(R'\setminus\{1\},Q')$ according to the ordering on $\cT_{k}$.

	First, we prove the following two properties of $M^{(1)}$.

	\begin{enumerate}
		\item The matrix $M^{(1)}$ is element-wise larger than or equal
			to $\cP_k$.

		\item The matrix $M^{(1)}$ is not equal to $\cP_k$.
	\end{enumerate}

	Since $|S|=|\cT_k|$, the matrix $M^{(1)}$ is a $|\cT_k|
	\times |\cT_k|$ matrix. We start with proving the second
	statement. The entry of $\cP_k$ in the intersection of the row
	and column that correspond to the type $(\{1,\ldots, 1\}, \emptyset)\in
	\cT_k$ and $(\emptyset,\{k\})\in \cT_k$, respectively,
	is clearly $1$. On the other hand, the value of $M^{(1)}$ in this row
	and column corresponds to the number of transitions from $(\{1,\ldots,
	1\},\emptyset)\in \cT_{k+1}$ to $(\{1\},\{k\})$ which is $k+1$
	(the number of ways one can choose one of the $k+1$ paths which will
	not be extended). Therefore $M^{(1)} \neq \cP_k$.

	For the first statement, for any two types $(R,Q),(R',Q') \in
	\cT_k$, if a type $(R,Q)$ subgraph of a $k$-flow on the
	vertices $U_\ell$ can be extended to an another type $(R',Q')$ subgraph
	on the vertices $U_{\ell+1}$ in $p$ ways, then clearly a type
	$(R\cup\{1\},Q)$ subgraph of a $k+1$-flow on the vertices
	$U_\ell$ can be extended to a type $(R'\cup \{1\},Q')$ subgraph on the
	vertices $U_{\ell+1}$ in at least $p$ ways. Therefore the first
	property is also proven.

	Suppose to the contrary that $r_{k+1} \leq r_k$. Since the limit
	\[
		\lim_{n \rightarrow \infty}\frac{\cP_{k+1}^n}{r_{k+1}^n}
	\]
	exists and is finite, both the limits
	\[
		\lim_{n \rightarrow
	\infty}\frac{\cP_{k+1}^n}{r_{k}^n} \quad \text{and} \quad
	\lim_{n \rightarrow \infty}\frac{M^{(n)}}{r_{k}^n}
	\]
	exist and are finite. Since $M^{(1)}$ is a principal minor of
	$\cP_{k+1}$, and every element of $\cP_{k+1}$ is
	non-negative, for all $k$ the matrix $M^{(k)}$ is element-wise larger
	than or equal to ${(M^{(1)})}^k$. Hence the sequence
	\[
		{\left\{ \frac{{(M^{(1)})}^n}{r_{k}^n} \right\}}_{n=1}^{\infty}
	\]
	is bounded. By the two properties of $M^{(1)}$ and the fact that
	$\cP_k$ is primitive, it follows that there is an integer $m$
	such that ${(M^{(1)})}^m$ is element-wise strictly larger than
	$\cP_k^m$. Thus there is a positive $\varepsilon$ such that
	${(M^{(1)})}^m$ is element-wise strictly larger than
	$(1+\varepsilon)\cP_k^m$. Therefore the sequence
	\[
		{\left\{\frac{{((1+\varepsilon)\cP_k^m)}^n}{r_{k}^{mn}}
		\right\}}_{n=1}^{\infty}= {\left\{
		{(1+\varepsilon)}^n\frac{\cP_k^{mn}}{r_{k}^{mn}}
		\right\}}_{n=1}^{\infty}
	\]
	is bounded, but this clearly contradicts the fact that the limit
	\[
		\lim_{n \rightarrow \infty}\frac{\cP_{k+1}^n}{r_{k+1}^n}
	\]
	is a one dimensional projection.
\end{proof}

\begin{proof}[Proof of Theorem~\ref{thm:nonstable}]
	Observe, that $\|\bm{h}_{k+1}(n)-\bm{h}_k(n)\|_1=2$. However, according
	to Lemma~\ref{lem:eigenval}
	\begin{equation*}
		\frac{|\cG(\bm{h}_{k+1}(n))|}
		{|\cG(\bm{h}_{k}(n))|} =
		\varTheta\left({\left(\frac{r_{k+1}}{r_k}\right)}^n\right),
	\end{equation*}
	which grows exponentially as $n\to \infty$,
	so $\cH_k$ is not $P$-stable.
\end{proof}

\section{Concluding remarks}\label{sec:conclusion}

\subsection{Relationship to prior results}

Although the sets of degree sequences $B_{2k}(\cH_0)$ (for some $k$)
are certainly not diverse compared to the class of $P$-stable degree sequences,
they are more numerous than, say, the regular degree sequences, for which rapid
mixing of the switch Markov chain were first proven
in~\cite{cooper_sampling_2007, miklos_towards_2013, greenhill_polynomial_2011}.
Because $B_{2k}(\cH_0)$ is not $P$-stable, the Jerrum-Sinclair
chain~\cite{jerrum_fast_1990} cannot produce a sample in polynomial expected
time. Although in principle, the proof of rapid mixing on $P$-stable degree
sequences~\cite{erdos_mixing_2019} may be applicable to
$B_{2k}(\cH_0)$, we do not expect that it can be easily tweaked to
accommodate it, for the following reasoning:

\medskip

Let $\cT$ be the set of $(X,Y)$ pairs of realizations of $\bm{h}_1(n)$
such that the paths $\vec\nabla(X)$ and $\vec\nabla(Y)$ are edge
disjoint. It is simple to show that $|\cT|\ge\exp(cn)\cdot
|\cG(\bm{h}_1(n))|$, because for almost every realization $X$ we have
$|E(\vec\nabla(X))|\approx \frac{2n}{\sqrt{5}}$. For a pair
$(X,Y)\in\cT$, the edges $E(X)\triangle E(Y)$ form a cycle which
traverses both $a_1$ and $b_n$. From this structure it follows that the
multicommodity flow $\Gamma$ described in~\cite{erdos_mixing_2019} between a
pair of realizations $(X,Y)\in\cT$ is a single switch sequence that
passes through $H_0(n)-a_1 b_n\in\cG(\bm{h}_1(n))$. Consequently, the
load $\rho(\Gamma)\ge |\cT|/|\cG(\bm{h}_1(n))|\ge \exp(cn)$
is exponential in $n$.

\subsection{Unconstrained (simple) graphs}\label{sec:reduction}
As mentioned in Section~\ref{sec:bip}, $\Psi^{-1}$ embeds splitted bipartite
graphs into the space of simple graphs. The map $\Psi^{-1}$ preserves
switches, since the symmetric difference of the edge sets of two realizations
does not change by adding a clique to both graphs. Consequently, $\Psi^{-1}$
induces an isomorphism between the Markov-graphs $\cM(\bm{d})$ and
$\cM(\bm{d})(\Psi^{-1}(\bm{d}))$.

Furthermore, through $\Psi^{-1}$, a set of canonical paths $\Gamma$ on
$G(\cM(\bm{d}))$ are mapped to a set of canonical paths $\Psi^{-1}(\Gamma)$
on $G(\cM(\Psi^{-1}(\bm{d})))$ satisfying
\begin{equation*}
	\rho(\Psi^{-1}(\Gamma))\le\rho(\Gamma).
\end{equation*}

In summary, Theorem~\ref{thm:main} can be pulled back to simple graphs: the
switch Markov chain is rapidly mixing on $\Psi^{-1}(B_{2k}(\cH_0))$. Note,
however, that
\begin{equation*}
	\Psi^{-1}(B_{2k}(\cH_0)) \subset B_{2k}(\Psi^{-1}(\cH_0)),
\end{equation*}
because the right hand side contains graphs that are not split.

\subsection{Possible generalizations}

The proof of Theorem~\ref{thm:main} presented in Section~\ref{sec:main} works
verbatim up to $k=\Theta(\sqrt{\log{n}})$, one just has to check the
dependence on $k$ in Section~\ref{sec:load}. In other words, the switch Markov
chain is rapidly mixing on
\begin{equation*}
	\bigcup_{n=1}^\infty B_{c\cdot\sqrt{\log n}}(\bm{h}_0(n))
\end{equation*}
for some $c>0$.  We have not proved nor refuted $P$-stability of
$\bigcup_{n=1}^\infty \bS_{2k}(\bm{h}_0(n))$ when $k=\Theta(\sqrt{\log{n}})$.

We hope that the proof of Theorem~\ref{thm:main} can be generalized to even
broader classes. A defining property of $\bm{h}_k(n)$ is that for any
realization $G\in \cG(\bm{h}_k(n))$ and $i\in[1,n]$, we have
\begin{equation*}
	\delta_{\vec\nabla(G)}(A_i\cup B_i)=\delta_{\vec\nabla(G)}(A_i\cup
	B_i\setminus \{b_i\})=k.
\end{equation*}
Relax these constraints to requiring only that $\leq k$
edges leave $A_i\cup B_i$ and $A_i\cup B_i\setminus\{b_i\}$ for every $i\in
[1,n]$: the set of graph satisfying these is the set of realizations of a set of
degree sequences we will call $\cH_{\le k}$. Naturally, $B_{2k}(\cH_0)\subseteq
\cH_{\le k}$, because a $k$-flow needs at most $k$ edges in any cut.

\begin{conjecture}
	For any fixed $k$, the switch Markov chain is rapidly mixing on
	$\cH_{\le k}$.
\end{conjecture}

We also put forward a conjecture inspired by the work Greenhill and
Gao~\cite{gao_mixing_2020}. Recall Definition~\ref{def:Pstable}.

\begin{conjecture}\label{conj:bip}
	Suppose $\cD$ is $\left(2k+2\right)$-stable for some $k\in \bN$. Then
	the switch Markov chain is rapidly mixing on
	$B_{2k}(\overline{\cD^\circ})$.
\end{conjecture}

\linespread{1}

\end{document}